\documentclass[a4paper, 11pt]{article}

\usepackage[
            includefoot,  
            marginparwidth=0in,     
            marginparsep=0in,       
            margin=1.45in,               
            includemp]{geometry}

\usepackage{bbm}
\usepackage{float}
\usepackage{graphicx}                  
\usepackage{amssymb}
\usepackage{amsfonts}
\RequirePackage{amsmath}
\RequirePackage{amsthm}

\usepackage{color}

\usepackage{fancyhdr}
\theoremstyle{montheoreme}

\newtheorem{thm}{Theorem}[section]

\newtheorem{prop}[thm]{Proposition}

\newtheorem{rque}{Remark} [section]

\theoremstyle{maremarque}

\DeclareMathOperator{\Var}{Var}

\newcommand{\R}{\mathbb{R}}

\newcommand{\cL}{\mathcal{L}}

\begin{document}

\title{Stability of Klartag's  improved Lichnerowicz inequality}
\date{\today}
\author{%
  Thomas~A.~Courtade\footnote{University of California, Berkeley, United States, Department of Electrical Engineering and Computer Sciences.\\
    courtade@berkeley.edu}
  \and 
  Max~Fathi\footnote{Université Paris Cité and Sorbonne Université, CNRS, Laboratoire Jacques-Louis Lions and Laboratoire de Probabilit\'es, Statistique et Mod\'elisation, F-75013 Paris, France\\
and DMA, École normale supérieure, Université PSL, CNRS, 75005 Paris, France \\
and Institut Universitaire de France\\ mfathi@lpsm.paris }}

\maketitle

\begin{abstract}
In a recent work, Klartag gave an improved version of Lichnerowicz' spectral gap bound for uniformly log-concave measures, which improves on the classical estimate by taking into account the covariance matrix. We analyze the equality cases in Klartag's bound, showing that it can be further improved whenever the measure has no Gaussian factor. Additionally, we give a quantitative improvement for log-concave measures with finite Fisher information. 
\end{abstract}

\section{Introduction and main results}

The Poincar\'e constant of a probability measure $\mu$ on $\R^d$, which we shall denote by $C_P \equiv C_P(\mu)$, is the smallest constant $C$ such that for any locally Lipschitz function $f$, we have
$$\Var_{\mu}(f) := \int{f^2d\mu} - \left(\int{fd\mu}\right)^2 \leq C\int{|\nabla f|^2d\mu}.$$
This inequality has many applications in probability, analysis and statistics, including concentration inequalities and estimates on rates of convergence for Markov processes. We refer to the monograph \cite{BGL14} for background on this inequality. 

The goal here is to study the equality cases and near-equality cases of the following result of \cite{Kla23}:

\begin{thm}[Improved log-concave Lichnerowicz theorem, Klartag 2023]
Let $\mu$ be a $t$-log-concave probability measure with covariance matrix $A$. Then
\begin{align}
C_P(\mu) \leq \sqrt{\frac{\|A\|_{op}}{t}}.\label{eq:ImprovedLichnerowicz}
\end{align}
\end{thm}

This theorem improves on the Lichnerowicz bound on the spectral gap  for log-concave measures (which states that $C_P \leq t^{-1}$ in the Euclidean setting), since for such measures the variance in any direction is always at most $t^{-1}$. This bound is not actually due to Lichnerowicz, the name stems from an analogy with his work on the spectral gap of positively curved Riemannian manifolds \cite{Lich, BGL14}. 

The Kannan-Lovasz-Simonovits conjecture predicts a dimension-free upper bound of the form $C_P \leq c \|A\|_{op}$ for numerical constant $c$, but there is no predicted value of $c$. The best result currently known is an upper bound of the form $c \leq \alpha \log d$,  where $\alpha$ is a numerical constant.  This was derived in \cite{Kla23} using the above improved Lichnerowicz bound. 

The equality cases in the original Lichnerowicz bound were shown to split off a Gaussian factor in \cite{CZ17, GKKO18}, in a broader geometric setting. A similar splitting result for the stronger logarithmic Sobolev inequality is also available in the geometric setting \cite{OT}. Our first result is the analysis of equality cases in Klartag's improved inequality, showing that equality holds only if $\mu$ is a product measure with a Gaussian factor in the direction of largest variance: 

\begin{thm}[Rigidity] \label{thm_rigid}
If $\mu$ is a $t$-log-concave probability measure satisfying \eqref{eq:ImprovedLichnerowicz} with equality, then $\mu$ splits off a Gaussian factor in a direction of largest variance.  More precisely, up to a rotation and translation, $\mu$ is of the form $\bar{\mu}\otimes   \gamma_{1/t} $, where $\bar{\mu}$ is the marginal of $\mu$ on the first $d-1$ coordinates and  $\gamma_{1/t}$ is the one-dimensional  Gaussian measure  with variance $1/t = \|A\|_{op}$. 
\end{thm}

We remark that the Poincar\'e constant $C_P(\mu)$ is invariant to translations and rotations of $\mu$, so the Gaussian direction cannot be identified without further assumptions. Measures which split off  a Gaussian factor in the  direction of maximum variance are also equality cases in the unimproved Lichnerowicz theorem \cite{CZ17, GKKO18}. Hence, while the improved Lichnerowicz theorem is almost always strictly better than the unimproved inequality, whenever it actually improves, it is not sharp. 

Recently, there has been much interest in quantitative stability for functional inequalities. The problem consists in identifying an explicit improvement in the original functional inequality, that involves some distance to the set of equality cases. We refer to the surveys \cite{Fig13, Fig23} for the broader context (including Sobolev inequalities, isoperimetric inequalities and the Brunn-Minkowski inequality as emblematic examples). For stability in the Lichnerowicz spectral gap bound in Euclidean space, the first results were obtained in \cite{DPF17} using regularity properties of the Monge-Amp\`ere equation, and were later improved in \cite{CF20} using Stein's method. There are also partial results in the geometric setting \cite{BF, FGS}. In the present work, we  also rely on  Stein's method --- indirectly, through the results of \cite{CF20} --- to study stability for Klartag's improved estimate. 

\begin{thm}[Stability] \label{thm_stab}
Let $\mu$ be a $t$-uniformly log-concave probability measure with covariance matrix $A$,  and   $\beta^2 := \sup_{v \in \mathbb{S}^{d-1}} \int{\langle \nabla V(x), v\rangle^2d\mu}$. There is an absolute constant $c>0$ (that in particular does not depend on $d$) such that, up to a rotation and translation, 
$$
C_P \leq \left( \frac{\|A\|^2_{op} (1+\beta)^2}{\|A\|^2_{op}(1+\beta)^2 + c ( \|A\|^2_{op} \wedge W_1(\mu, \bar{\mu} \otimes \gamma_{1/t})^4) } \right)  \sqrt{\frac{\|A\|_{op}}{t}},
$$
where $\bar{\mu}$ is the marginal of $\mu$ on the first $d-1$ coordinates and $\gamma_{1/t}$ is the Gaussian measure with variance $1/t$. 
\end{thm}

\begin{rque}
We can trivially bound $\beta^2$ by the Fisher information $\int{|\nabla V|^2d\mu}$ (and also control it by the Fisher information relative to the Gaussian, up to an additive constant), but this comparison typically loses a dimensional factor. 
\end{rque}

The sequel consists in the proofs of these two theorems. 

\textbf{Acknowledgments}: This work was done during the workshop \emph{Interactions between Probability and PDE} at the CIRM in 2023, and we thank Joseph Lehec for his lectures on \cite{Kla23} that inspired this work. MF was supported by the Agence Nationale de la Recherche (ANR) Grant ANR-23-CE40-0003 (Project CONVIVIALITY).   TC acknowledges NSF-CCF~1750430, the hospitality of the Laboratoire de Probabilit\'es, Statistique et Mod\'elisation (LPSM) at the Universit\'e Paris Cit\'e, and the  Invited Professor program of the Fondation Sciences Math\'ematiques de Paris (FSMP). 

\section{Proofs}

In the sequel, the covariance matrix of a probability measure is defined by
$$\operatorname{Cov}(\mu)_{ij} := \int{x_ix_jd\mu} - \left(\int{x_id\mu}\right)\left(\int{x_jd\mu}\right).$$
It is a well-defined matrix when $\mu$ is a log-concave probability measure, and is symmetric positive semidefinite.

A probability measure is said to be $t$-log-concave for some $t \in \R$ if it has a density $e^{-V}$ with respect to the Lebesgue measure, and such that the function $x\mapsto V(x) - t|x|^2/2$ is convex. We shall only consider the case $t > 0$ here. 

Given a $t$-log-concave probability measure $\mu$, we define the linear operator $\cL$ by the relation
$$\int{(\cL f)gd\mu} = -\int{\nabla f \cdot \nabla g d\mu}$$
for all smooth compactly supported functions $f,g$. When $\mu$ has a positive density $e^{-V}$ on its support (which is the case in non-degenerate situations),  $\cL$ is given by the formula
$$\cL f = \Delta f - \nabla V \cdot \nabla f.$$
The operator $\cL$ can   be viewed as the generator of a drift-diffusion process that is reversible with respect to $\mu$. The measure $\mu$ satisfies a Poincar\'e inequality iff $\cL$ has a spectral gap, and $C_P^{-1}$ is the infimum of the positive eigenvalues of $-\cL$.

While   the spectral gap might not be attained in general, the spectral gap is a true eigenvalue of the generator  for uniformly log-concave measures. That is, there exists a non-zero eigenfunction $f$ such that $\cL f = - C_P^{-1}f$ (as a consequence of \cite[Proposition 6.7]{GMS}, for example). 

\subsection{Proof of Theorem \ref{thm_rigid}}
The proof consists in keeping track of equality cases in the proof of the improved Lichnerowicz inequality \cite{Kla23}, to see that an eigenfunction associated with the spectral gap $C_P(\mu)^{-1}$ is affine, which then forces equality in the original Lichnerowicz inequality. A more general result of E. Meckes \cite{Mec} also states that,  even without log-concavity, if an eigenfunction of a Markov diffusion generator is affine, then the measure is Gaussian in the associated direction.  However, this result of Meckes  does not automatically imply independence of the Gaussian factor relative to orthogonal directions.  

\begin{proof}
The Poincar\'e inequality is equivalent to the inequality
$$\int{|\nabla g|^2d\mu} \leq C_P(\mu)\int{(\cL g)^2d\mu}.$$
By the Bochner formula
\begin{align} \label{bochner}
\int{(\cL g)^2d\mu} &= \int{\|\nabla^2 g\|_{HS}^2d\mu} + \int{\langle \nabla^2 V \nabla g, \nabla g\rangle d\mu} \notag\\
&\geq \int{\|\nabla^2 g\|_{HS}^2d\mu} +t \int{|\nabla g|^2 d\mu} \notag \\
&\geq C_P^{-1}\left(\int{|\nabla g|^2d\mu} - \left|\int{\nabla g d\mu}\right|^2\right) +t \int{|\nabla g|^2 d\mu}.
\end{align}
Without loss of generality, we can rescale the measure $\mu$ so that $\|A\|_{op} = 1$. After rescaling, the value of $t$ is necessarily smaller than $1$ (for example, by testing the Lichnerowicz bound on $C_P(\mu)$ with affine functions). The goal is to show that if $C_P = t^{-1/2}$ then $t = 1$. This will imply that equality cases in the improved Lichnerowicz inequality are also equality cases in the original Lichnerowicz inequality, and then we will just apply the rigidity result of \cite{CZ17}. 

Assume that $f$ is an eigenfunction of $-\cL$ with eigenvalue $t^{1/2}$, normalized so that $\int{f^2d\mu} = 1$. Note that we necessarily have $\int{f d\mu} = 0$ since $f$ is orthogonal to the kernel of $\cL$, which is the set of constant functions. We also have
$$\int{|\nabla f|^2d\mu} = -\int{(\cL f)fd\mu} =  \sqrt{t}; \hspace{3mm} \int{(\cL f)^2d\mu} =  t.$$

If $\int{\nabla f d\mu} = 0$ then taking $g = f$ in \eqref{bochner} would yield $t \geq t  + t^{3/2}$ which would be a contradiction. So without loss of generality we can assume $\int{\nabla f d\mu} \neq 0$. Let 
$\theta = \frac{\int{\nabla f d\mu}}{\left|\int{\nabla f d\mu}\right|}$. 
We have
\begin{align}
\left|\int{\nabla fd\mu}\right|^2 &= \left(\int{\nabla f \cdot \theta d\mu}\right)^2 \notag \\
&= \left(\int{\nabla f \cdot \nabla (\langle \theta, x \rangle)d\mu}\right)^2 \notag \\
&= \left(\int{(\cL f)\langle \theta, x\rangle d\mu}\right)^2 \notag \\
&\leq \int{(\cL f)^2d\mu} \int{\langle \theta, x\rangle^2d\mu} \label{cs_improved_lich} \\
&\leq \int{(\cL f)^2d\mu}\|A\|_{op}|\theta|^2  =  \int{(\cL f)^2d\mu}= t. \label{direction_improved_lich}
\end{align}
Combined with \eqref{bochner} we get $t \geq t - t^{3/2} +t^{3/2}$. But since there is equality, there is actually equality at every step along the way, and in particular in \eqref{cs_improved_lich} and \eqref{direction_improved_lich}.

Equality in \eqref{direction_improved_lich} means that $\theta$ is a direction of unit variance (i.e. maximal in the scaling we consider). Equality in \eqref{cs_improved_lich} means that $\cL f(x) = \lambda \langle x, \theta\rangle$ for some $\lambda \in \R$. Equating the $L^2$ norms of both sides identifies $\lambda = \pm \sqrt{t}$, and since $f$ is an eigenfunction, we have $f(x) = \pm \langle \theta, x\rangle$. But then $\nabla f = \pm \theta$ and therefore $\sqrt{t} = \int{|\nabla f|^2d\mu} = 1$ so that $t = 1 = \|A\|_{op}$, and also $C_P(\mu) = 1$. The conclusion follows from the equality case in the unimproved Lichnerowicz theorem on $\R^d$ \cite[Theorem 2]{CZ17}, and rescaling.

\end{proof}

\subsection{Proof of Theorem \ref{thm_stab}}
We will show that near-equality in the improved Lichnerowicz inequality forces $t$ to be close to $1$, which forces near-equality in the unimproved Lichnerowicz theorem.  The conclusion will then follow from the results of \cite{CF20}, which we quote here in a suitable form.
\begin{prop}[{\cite[Theorem 1.2]{CF20}}]\label{prop:stabilityPI}
Let $\mu$ be a $t$-uniformly log-concave probability measure.  If $C_P^{-1}\leq t  (1+\delta)$, then  then up to a rotation and translation
$$W_1(\mu, \bar{\mu} \otimes \gamma_{1/t}) \leq 26 \sqrt{\delta},$$
where $\bar{\mu}$ is the projection of $\mu$ onto the first $d-1$ coordinates and $\gamma_{1/t}$  is the  Gaussian measure on the real line with variance $1/t$.  
\end{prop}

\begin{proof}[Proof of Theorem   \ref{thm_stab}]
Assume that $\mu$ is centered and rescaled so that $\|A\|_{op}=1$.    Begin by defining $\epsilon \geq 0$ via $C_P^{-1} = \sqrt{t}(1+\epsilon)$.   Consider an eigenfunction $f$ associated with the spectral gap (that is $\cL f = -t^{1/2}(1+ \epsilon)f$), normalized so that $\int{f^2d\mu} = 1$. As before, we  have $\int{f d\mu} = 0$ since $f$ is orthogonal to the kernel of $\cL$, which is the set of constant functions.   

Testing \eqref{bochner} with $g = f$, we have
$$t(1+\epsilon)^2 \geq t(1+\epsilon)^2 + t^{3/2}(1+\epsilon) - \sqrt{t}(1+\epsilon)\left|\int{\nabla fd\mu}\right|^2$$
and hence
$$\left|\int{\nabla fd\mu}\right|^2 \geq t.$$ 
In particular, $|\int{\nabla fd\mu}|\neq 0$, and $\theta = \frac{\int{\nabla f d\mu}}{\left|\int{\nabla f d\mu}\right|}$ is well-defined.  Put 
$$A := 1-\Var_{\mu}(\langle x, \theta\rangle); \hspace{5mm} B := \int{(\cL f)^2d\mu} \int{\langle \theta, x\rangle^2d\mu} - \left(\int{\cL f\langle x, \theta\rangle d\mu}\right)^2.$$
Arguing as in  \eqref{cs_improved_lich}, we have
$$
B + t(1+\epsilon)^2 A \leq t \left((1+\epsilon)^2-1 \right) .
$$
Both $A$ and $B$ are nonnegative.  In particular, 
$$
1 - A \geq   \frac{1}{(1+\epsilon)^2} ; ~~~B \leq t \left((1+\epsilon)^2-1 \right).
$$
As is standard in a Hilbert space with norm $\|\cdot\|$, we have
$$\inf_{\lambda \in \R} \|x-\lambda y\|^2 \leq \|x\|^2 - \frac{\langle x, y \rangle}{\|y\|^2}.$$
Using this inequality with the $L^2(\mu)$ Hilbert structure, we get
$$\inf_\lambda \int{(\cL f - \lambda \langle x, \theta\rangle)^2d\mu} \leq \frac{B}{\int{\langle x, \theta\rangle^2d\mu}} = \frac{B}{1-A} \leq t (1+\epsilon)^2   \left((1+\epsilon)^2-1 \right) .$$
Denoting by $\lambda_0$ the optimal $\lambda$, we have
$$\|\cL f - \lambda_0\langle x, \theta\rangle\|_{L^2(\mu)}^2 \leq  t (1+\epsilon)^2   \left((1+\epsilon)^2-1 \right) .$$
In particular, by the triangle inequality
$$|(1+\epsilon)\sqrt{t} -\lambda_0\sqrt{1-A}|^2 \leq  t (1+\epsilon)^2   \left((1+\epsilon)^2-1 \right) $$
so that another application of the triangle inequality gives
$$  |\lambda_0 - \sqrt{t}| \leq   \sqrt{t}  \left(  (1+\epsilon)^2   \sqrt{(1+\epsilon)^2-1} + (1+\epsilon)^2-1 \right), $$
where we used $\sqrt{1-A}\geq (1+\epsilon)^{-1}$.
Therefore, 
\begin{align*}
\|\cL f - \sqrt{t}\langle x, \theta\rangle\|_{{L^2(\mu)}} &\leq \|\cL f -\lambda_0 \langle x, \theta\rangle\|_{L^2(\mu)} 
+  |\lambda_0 - \sqrt{t}|  \| \langle x,  \theta\rangle\|_{{L^2(\mu)}}\\
&\leq  2  \sqrt{t}  \left(  (1+\epsilon)^2   \sqrt{(1+\epsilon)^2-1} + (1+\epsilon)^2-1 \right),
\end{align*}
and we conclude
\begin{align*}
\sqrt{t} \|  f +  \langle x, \theta\rangle\|_{{L^2(\mu)}} &\leq \sqrt{t} \|(1+\epsilon)   f +  \langle x, \theta\rangle\|_{{L^2(\mu)}}
+\sqrt{t}\epsilon  \|   f   \|_{{L^2(\mu)}} \\
&=
\|\cL f - \sqrt{t}\langle x, \theta\rangle\|_{{L^2(\mu)}} + \sqrt{t}\epsilon\\
&\leq   3  \sqrt{t}  \left(  (1+\epsilon)^2   \sqrt{(1+\epsilon)^2-1} + (1+\epsilon)^2-1 \right).
\end{align*}
Dividing through by $\sqrt{t}$, we can estimate 
\begin{align} 
\|\nabla f - \theta\|_{L^2(\mu)}^2 &= -\int{(\cL f + \cL \langle x, \theta\rangle)(f+\langle x, \theta\rangle )d\mu} \notag \\
&\leq (\|\cL f\|_{L^2(\mu)} + \|\cL \langle x, \theta \rangle\|_{L^2(\mu)})\|f - \langle x, \theta\rangle\|_{L^2(\mu)} \notag \\
&\leq 3 (t^{1/2}(1+\epsilon) + \|\cL \langle x, \theta \rangle\|_{L^2(\mu)})    \left(  (1+\epsilon)^2   \sqrt{(1+\epsilon)^2-1} + (1+\epsilon)^2-1 \right) . \label{est_l2_grad}
\end{align}

Moreover, since $\cL \langle x, \theta\rangle = \langle \nabla V(x), \theta\rangle$, we have
\begin{equation}
\|\cL \langle x, \theta \rangle\|_{L^2(\mu)}^2 = \int{\langle \nabla V(x), \theta\rangle^2d\mu} \leq \beta^2.
\end{equation}
Combined with \eqref{est_l2_grad} and $t \leq 1$, we get
\begin{align*}
\|\nabla f - \theta\|_{L^2(\mu)}^2 &\leq 3(1+\epsilon + \beta) \left(  (1+\epsilon)^2   \sqrt{(1+\epsilon)^2-1} + (1+\epsilon)^2-1 \right).
\end{align*}
Let's now consider what happens when $\sqrt{\epsilon} \leq 1/(25 (1+\beta))$.  In this case, the above can be crudely bounded as
\begin{align*}
\|\nabla f - \theta\|_{L^2(\mu)}^2 &\leq 12 (1+\beta) \sqrt{\epsilon}\leq  \frac{1}{2}.
\end{align*}
But then
$$\sqrt{t}(1+\epsilon) = \|\nabla f\|_{L^2(\mu)}^2 \geq (|\theta| - \|\nabla f - \theta\|_{L^2(\mu)})^2 \geq 1 - 12 (1+\beta) \sqrt{\epsilon} \geq \frac{1}{2}, $$
so that 
$$
C_P^{-1} =  \sqrt{t}(1+\epsilon) \leq t \frac{ (1 +  \epsilon)^2 }{ 1 - 12 (1  + \beta)    \sqrt{\epsilon}   }.
$$
Hence, we may take 
$$
\delta = \frac{12 (1 + \beta)  \sqrt{\epsilon}   + 2 \epsilon + \epsilon^2 }{ 1 - 12 (1 + \beta)   \sqrt{\epsilon}    }
\leq  26(1 + \beta)  \sqrt{\epsilon}    
$$
in Proposition \ref{prop:stabilityPI} to conclude, up to rotation and translation (which we assume henceforth), 
$$
W_1(\mu, \bar{\mu} \otimes \gamma_{1/t})^2 \leq (26)^3(1 + \beta)  \sqrt{\epsilon} .
$$
Hence,  there is a universal constant $c>0$ such that 
$$
\frac{c}{(1+\beta)^2}W_1(\mu, \bar{\mu} \otimes \gamma_{1/t})^4 \leq \epsilon, 
$$
and using $C_P^{-1} = t^{1/2}(1+\epsilon)$, 
$$
C_P \leq \left( \frac{(1+\beta)^2}{(1+\beta)^2 + c W_1(\mu, \bar{\mu} \otimes \gamma_{1/t})^4 } \right)  \frac{1}{\sqrt{t}}.
$$
In the complementary case where $\sqrt{\epsilon} > 1/(25 (1+\beta))$, we may again use $C_P^{-1} = t^{1/2}(1+\epsilon)$ to write 
$$
C_P \leq \left( \frac{(1+\beta)^2}{(1+\beta)^2 + (25)^{-2} } \right)  \frac{1}{\sqrt{t}}.
$$
We conclude that there is a universal constant $c>0$ such that, up to rotation and translation of $\mu$, 
$$
C_P \leq \left(  \frac{ (1+\beta)^2}{ (1+\beta)^2 + c ( 1 \wedge W_1(\mu, \bar{\mu} \otimes \gamma_{1/t})^4) } \right)  \sqrt{\frac{1}{t}}.
$$
Since  $\beta$ is invariant to dilation of  $\mu$,   the claim follows by rescaling.  
\end{proof}

\end{document}